\documentclass[11 pt]{article}
\usepackage{amsthm,amsmath,array,amssymb,amscd,amsfonts,latexsym,wasysym,mathtools}
\usepackage{bbold}
\usepackage[colorlinks,allcolors=blue]{hyperref}
\usepackage{url}
\usepackage{cite}
\usepackage{color}
\usepackage[dvipsnames]{xcolor}
\usepackage{graphicx}
\usepackage{tikz}
\usetikzlibrary{decorations.pathreplacing}
\usepackage{ytableau}
\usepackage{listings}
\usepackage{rotating}

\newtheorem{theo}{Theorem}[section]
\newtheorem{prop}[theo]{Proposition}

\newtheorem{lemm}[theo]{Lemma}

\newtheorem{coro}[theo]{Corollary}
\newtheorem{rema}[theo]{Remark}
\newtheorem{Defi}[theo]{Definitions}

\newtheorem{example}[theo]{Example}

\title{\bf A generalization of the Murnaghan-Nakayama rule for $K$-$k$-Schur and $k$-Schur functions}
\author{Duc-Khanh Nguyen}

\date{}

\newfont{\gothic}{eufb10}

\begin{document}
\maketitle
\begin{abstract} 
The $K$-$k$-Schur functions and $k$-Schur functions appeared in the study of $K$-theoretic and affine Schubert Calculus as polynomial representatives of Schubert classes. In this paper, we introduce a new family of symmetric functions $\mathcal{F}_\lambda^{(k)}$, that generalizes the constructions via the Pieri rule of $K$-$k$-Schur functions and $ k$-Schur functions. Then we obtain the Murnaghan-Nakayama rule for the generalized functions. The rule is described explicitly in the cases of $K$-$k$-Schur functions and $k$-Schur functions, with concrete descriptions and algorithms for coefficients. Our work recovers the result of Bandlow, Schilling, and Zabrocki for $k$-Schur functions, and explains it as a degeneration of the rule for $K$-$k$-Schur functions. In particular, many other special cases and connections promise to be detailed in the future.  
\end{abstract}

\section{Introduction} 
The classical Murnaghan-Nakayama rule is a combinatorial rule for computing the irreducible character values $\chi_\lambda(\mu)$ of the symmetric group using ribbon tableaux \cite{nakayama1940someII,nakayama1940someI,murnaghan1937characters,littlewood1934group}. In terms of symmetric functions, it can be understood as the decomposition of $p_r.s_\lambda$ to the sum of $s_\mu$, where $p_r$ is the power sum symmetric function, $s_\lambda$ is the Schur function associated to partition $\lambda$. In general, studying the decomposition rules (Monk rule, Pieri rule, Murnaghan-Nakayama rule, Littlewood-Richardson rule) of symmetric functions is very important because it lies at the crossroads of many different domains such as Representation Theory, Schubert Calculus, Algebraic Combinatorics. The Schur functions $s_\lambda$ are characters of polynomial irreducible representations of the general linear group $GL(n)$, and it is the polynomial representative of Schubert classes of the cohomology $H^*(G_{m,n})$ of Grassmannian $G_{m,n}$ \cite{MR1852463}. The product of two symmetric functions may tell us the direct sum decomposition into irreducible modules of a tensor product of two irreducible modules, or the number of certain geometric objects. For this reason, many generalizations of classical Murnaghan-Nakayama
rule were studied, for instance  \cite{tuan2021murnaghan,morrison2018two,tewari2016murnaghan,lubeck2016murnaghan,ross2014loop,konvalinka2012skew,bandlow2011murnaghan,halverson1998murnaghan,fomin1998noncommutative,halverson1995q}.\\

The $k$-Schur functions $s^{(k)}_\lambda$ are a variation of Schur functions $s_\lambda$. They were first introduced in \cite{lapointe2003tableau} to study Macdonald polynomials \cite{macdonald1998symmetric}, and then appeared in the study of affine Schubert Calculus. Namely, they are the polynomial representatives of Schubert classes of the homology $H_*(Gr)$ of the affine Grassmannian $Gr = SL_{k+1}(\mathbb{C}((t)))/SL_{k+1}(\mathbb{C}[[t]])$ \cite{lam2008schubert}. The Pieri rule \cite{lapointe2007k} provides a way to define $k$-Schur functions, and it is an important key to obtaining the Murnaghan-Nakayama rule \cite{bandlow2011murnaghan}.\\

The $K$-$k$-Schur functions $g_\lambda^{(k)}$ are a variation of $k$-Schur functions $s_\lambda^{(k)}$ in $K$-theoretic Schubert Calculus. Namely, when replacing homology $H_*(Gr)$ by $K$-homology $K_*(Gr)$, the Schubert classes are represented by $K$-$k$-Schur functions. The polynomials $g_\lambda^{(k)}$ are simultaneously introduced and characterized via the Pieri rule in \cite{lam2010k, morse2012combinatorics}, with an explicit combinatorial formula given recently by \cite{blasiak2022k}. The $K$-$k$-Schur functions carry continued interest in the mathematical community. For instance, in \cite{blasiak2022k}, the authors introduced symmetric functions called Katalan functions and proved that the $K$-$k$-Schur functions form a subfamily of the Katalan functions. They also conjectured that another subfamily of Katalan functions called the closed $k$-Schur Katalan functions are identified with the Schubert structure sheaves in the $K$-homology of the affine Grassmannian. The conjecture is verified in \cite{ikeda2022closed}.\\

In this paper, we introduce a new family of symmetric functions $\mathcal{F}_\lambda^{(k)}$, that generalizes the constructions via the Pieri rule of $K$-$k$-Schur functions and $k$-Schur functions (Section \ref{noncomfunction}). Then we obtain the Murnaghan-Nakayama rule for the generalized functions (Theorem \ref{MN_Fwk}). In the case of $K$-$k$-Schur functions and $k$-Schur functions, we describe explicitly the rule with concrete description and algorithms for the coefficients (Corollaries \ref{MN_gwk}, \ref{MN_gwk_partition}, Proposition \ref{K_lambda_mu_r} for $g_\lambda^{(k)}$ and Corollaries \ref{MN_swk}, \ref{MN_swk_partition}, Proposition \ref{k_lambda_mu_r} for $s_\lambda^{(k)}$). Our result for $s_{\lambda}^{(k)}$ recovers \cite{bandlow2011murnaghan} and explains it as a degeneration of the rule for $K$-$k$-Schur functions. \\

The proof of the main results is carried out as follows: For Theorem \ref{MN_Fwk}, the key point is the cancellation-free expression of the noncommutative power sum symmetric polynomials $p_r$ (Lemma \ref{p_hook}). This formula is directly obtained from the definition of $p_r$ with the help of the previous reduction steps (Lemmas \ref{he_hook}, \ref{s_hook}). The arguments in the proofs of the lemmas are based on the properties of weak hook words (Lemma \ref{left_right_cancel}). For the consequences on $K$-$k$-Schur functions, Corollary \ref{MN_gwk} is deduced directly from Theorem \ref{MN_Fwk}. Corollary \ref{MN_gwk_partition} is a translation of Corollary \ref{MN_gwk} in terms of partitions. The observations between them describe the formation of a partition that would appear in the Murnaghan-Nakayama rule. By this, we know which partitions will appear with nonzero coefficients in Corollary \ref{MN_gwk_partition}, and an algorithmic description to compute exactly the coefficients in Proposition \ref{K_lambda_mu_r}. For the consequences on $k$-Schur functions, the story is similar but simpler, the process of forming a partition during the transfer from Corollary \ref{MN_swk} to Corollary \ref{MN_swk_partition} is considered as the process from Corollary \ref{MN_gwk} to Corollary \ref{MN_gwk_partition}, but has been degenerated. So the results for the $k$-Schur functions are degenerations of the results for the $K$-$k$-Schur functions.\\

There are some directions to study from our work in the future. Namely, many other special or related cases of $\mathcal{F}_\lambda^{(k)}$ promise to be detailed, for instance, dual $k$-Schur functions \cite{lam2006affine}, affine stable Grothendieck polynomials \cite{lam2010k,lam2006affine}, closed $K$-$k$-Schur functions \cite{ikeda2022closed, takigiku2019pieri}, etc. Furthermore, one can look for the translations of those results via Peterson isomorphism that sends homology $H_*(Gr)$ to quantum cohomology $QH(Fl_n)$ \cite{lam2010quantum, Peterson}, or via $K$-theoretic Peterson isomorphism that sends $K$-homology $K_*(Gr)$ to quantum $K$-cohomology $QK(Fl_n)$ \cite{ikeda2020peterson}.\\

The paper is organized as follows: Sections \ref{Grassmannian_elements}--\ref{bij} present fundamental objects (Grassmannian elements, $(k+1)$-cores, and $k$-bounded partitions) which parametrize $K$-$k$ Schur functions, $k$-Schur function and the generalized symmetric functions $\mathcal{F}_\lambda^{(k)}$, and bijections between them. Section \ref{A_k} introduces a generalized algebra $\mathcal{A}_k$ that captures $0$-Hecke algebras and nilCoxeter algebras as quotient algebras. The weak hook words in this algebra are introduced. Lemma \ref{left_right_cancel} collects important properties of those words to help us obtain cancellation-free expression of noncommutative power sum symmetric polynomials in Section \ref{noncomfunction}. Section \ref{antihook} studies anti-weak hook words. We recall Edelman-Greene insertion to transform an anti-weak hook word to a weak hook word. The relation is packaged into Lemma \ref{to_hook_word}, and will be used later to translate the Murnaghan-Nakayama rule in terms of partitions in Sections \ref{MNKk}, \ref{MNk}. Section \ref{noncomfunction} introduces noncommutative symmetric functions, including the generalized functions $\mathcal{F}_\lambda^{(k)}$, in algebra $\mathcal{A}_k$. The cancellation-free expression of power sum symmetric polynomials is in Lemma \ref{p_hook}.  Section \ref{genMN} states Murnaghan-Nakayama rule for $\mathcal{F}_\lambda^{(k)}$. Sections \ref{MNKk}, \ref{MNk} describe explicitly the rule for $K$-$k$-Schur functions and $k$-Schur functions.\\   

\textbf{Acknowledgments:} 
The author would like to express his gratitude to Prof. Cristian Lenart and Prof. Satoshi Naito to introduce related topics that led to this work. He is grateful to Prof. Anne Schilling and Prof. Mike Zabrocki for explanations in \cite{bandlow2011murnaghan}. The work is sponsored by DAAD postdoctoral fellowship at Otto-von-Guericke-Universität Magdeburg, Germany. The author would like to thank Prof. Petra Schwer for her support during his visit, and Prof. Rebecca Waldecker for her influenced questions in presentation and writing. We are grateful to the referees for their extensive knowledge, valuable comments, and encouragement, which helps to improve the text a lot.

\section{Affine symmetric group}\label{Grassmannian_elements}
\begin{Defi}
\normalfont
Fix $k \in \mathbb{Z}_{\geq 0}$. Let $\widetilde{S}_{k+1}$ be the {\em affine symmetric group} with generators $\{s_0,\dots,s_k\}$ satisfying relations 
\begin{align}
    s_i^2 &= 1 \text{ for all $i$,} \label{R1} \\
    s_is_{i+1}s_i &= s_{i+1}s_is_{i+1} \text{ for all $i$,} \label{R2}\\
    s_is_j &= s_js_i \text{ for all $i-j \ne \pm 1$,} \label{R3}
\end{align}
where the indices are taken from $\mathbb{Z}/(k+1)\mathbb{Z}$. For $w \in \widetilde{S}_{k+1}$, a shortest expression of $w$ in generators is called a {\em reduced word}. The {\em length} $l(w)$ of $w$ is defined to be the length of its reduced word. We write $s_{i_1 \dots i_m}$ for $s_{i_1} \dots s_{i_m}$. The {\em finite symmetric group} $S_{k+1}$ is the subgroup of $\widetilde{S}_{k+1}$ generated by $\{s_1,\dots,s_k\}$. Set $\widetilde{S}^0_{k+1}$ to be the set of minimal length coset representatives of $\widetilde{S}_{k+1}/S_{k+1}$. The elements of $\widetilde{S}^{0}_{k+1}$ are said {\em Grassmannian elements}. 
\end{Defi}
\begin{example}\label{wlambda} 
\normalfont
Let $k=4$, then $s_{23043210}$ is an element in $\widetilde{S}^{0}_{k+1}$. 
\end{example}
\section{\texorpdfstring{$(k+1)$-}.cores and \texorpdfstring{$k$-}.bounded partitions}
\begin{Defi}
\normalfont
A {\em partition} $\lambda$ is a weakly decreasing sequence of non-negative integers $(\lambda_1,\dots, \lambda_l)$. Its {\em length} $l(\lambda)$ is defined by the last index $l'$ such that $\lambda_{l'} >0$. Its {\em size} is defined by $|\lambda| = \lambda_1 +\dots+\lambda_l$. 
We identify $\lambda$ with its {\em Young diagram}, which is a collection of boxes arranged in left-justified rows, with $\lambda_i$ boxes in the $i$th row from the top. The {\em conjugate partition} $\lambda^t$ of $\lambda$ is obtained by reflecting $\lambda$ about its main diagonal. A box in the $i$th row and $j$th column is called a {\em cell} $(i,j)$. We say that it is a {\em $\lambda$-removable corner} if 
\begin{equation*}\label{removable_corner}
    (i,j) \in \lambda, \text{ and } (i+1,j), (i,j+1) \not\in \lambda,
\end{equation*}
and a {\em $\lambda$-addable corner} if 
\begin{align*}\label{addable_corner}
(i,j) &= (1,\lambda_1 + 1), \text{or }\\
(i,j) &= (l(\lambda)+1,1), \text{or }\\
(i,j) &\not\in \lambda, (i-1,j), (i,j-1) \in \lambda.
\end{align*}
The {\em hook-length} of a cell $(i,j)$ in $\lambda$ is the number of cells in $\lambda$, below or to the right of $(i,j)$, including $(i,j)$ itself.\\

Let $\mu=(\mu_1,\dots,\mu_l)$ be another partition. We say that $\mu \geq \lambda$ if $\mu_i \geq \lambda_i$ for all $i$. If $\mu \geq \lambda$, we define the {\em skew Young diagram} $\mu/\lambda$ to be the shape consisting of all boxes in $\mu$ but not in $\lambda$. The {\em size} of $\mu/\lambda$ is defined by $|\mu/\lambda| = |\mu|-|\lambda|$. A skew Young diagram is called a {\em ribbon} if it does not contain any $2\times 2$ square, a {\em vertical strip} if it does not contain two boxes in the same row, and a {\em horizontal strip} if it does not contain two boxes in the same column. We say that a skew shape is {\em connected} if, for any consecutive rows, there exist two cells in the same column. The {\em height $ht(\mu/\lambda)$} of a ribbon $\mu/\lambda$ is defined by the number of occupied rows minus the number of connected components. A {\em tableau} of skew shape $\mu/\lambda$ is a filling of skew diagram $\mu/\lambda$ by integers.\\

Fix $k \in \mathbb{Z}_{\geq 0}$. The {\em $(k+1)$-residue} of a cell $(i,j)$ is $j-i \mod (k+1)$. We define $supp(\mu/\lambda)$ to be the set of $(k+1)$-residues of cells in $\mu/\lambda$. A partition is called a {\em $(k+1)$-core} if it has no cell of hook-length $k+1$. The set of all $(k+1)$-cores is denoted by $\mathcal{C}_{k+1}$. A partition $\lambda$ is called {\em $k$-bounded} if $\lambda_1 \leq k$. The set of all $k$-bounded partitions is denoted by $\mathcal{P}_k$.
\end{Defi}

\begin{example}\label{lambda_kappa} 
\normalfont
Let $k=4$, then $\lambda = (4,2,1,1) \in \mathcal{P}_k$, $\kappa = (6,2,1,1) \in \mathcal{C}_{k+1}$. The tableau of $(k+1)$-residues of $\pi$ is marked red in this picture
$$\begin{array}{c|cccccccccc}
    &1&2&3&4&5&6&7&8&9\\
    \hline
     1&\color{red}{0}&\color{red}{1}&\color{red}{2}&\color{red}{3}&\color{red}{4}&\color{red}{0}&\color{blue}{1}&{2}&{3}\\
     2&\color{red}{4}&\color{red}{0}&\color{blue}{1}&{2}&{3}&{4}&{0}&{1}&{2}\\
     3&\color{red}{3}&\color{blue}{4}&\color{blue}{0}&{1}&{2}&{3}&{4}&{0}&{1}\\
     4&\color{red}{2}&{3}&{4}&{0}&{1}&{2}&{3}&{4}&{0}\\
     5&\color{blue}{1}&{2}&{3}&{4}&{0}&{1}&{2}&{3}&{4}\\
     6&\color{blue}{0}&{1}&{2}&{3}&{4}&{0}&{1}&{2}&{3}\\
     7&{4}&{0}&{1}&{2}&{3}&{4}&{0}&{1}&{2}\\
\end{array}$$
Here, the numbers on the top and the rightmost are indices of columns and rows, respectively. 
The set of $\kappa$-removable corners is $\{(4,1),(2,2),(1,6)\}$. The set of $\kappa$-addable corners is $\{(5,1), (3,2), (2,3), (1,7)\}$. Let $\pi = (7,3,3,1,1,1)$, then $\pi/\kappa$ is the ribbon shaped by blue entries, with support $\{0,1,4\}$.
\end{example}

\section{Bijections}\label{bij}
\begin{Defi}\label{def_bij}
\normalfont
Fix $k \in \mathbb{Z}_{\geq 0}$. By \cite{lapointe2005tableaux}, we have bijections
\begin{equation}\label{bijection}
    \mathcal{P}_k \xrightarrow[]{\phi} \widetilde{S}^0_{k+1} \xrightarrow[]{\mathfrak{s}} \mathcal{C}_{k+1}\xrightarrow[]{\mathfrak{p}} \mathcal{P}_k ,
\end{equation}
with $ \mathfrak{p} \circ \mathfrak{s} \circ \phi = \mathbb{1}$. They are described as follows. 
\begin{itemize}
    \item[1.] $\phi: \mathcal{P}_k \rightarrow \widetilde{S}^0_{k+1}, \lambda \mapsto w_\lambda$, where $w_\lambda = s_{i_1 \dots i_m}$ with $(i_1, \dots, i_m)$ is obtained by reading the $(k+1)$-residue of cells in $\lambda$ from the bottom row to the top row, and in each row from right to left.

    \item[2.] $\mathfrak{s}: \widetilde{S}^0_{k+1} \rightarrow \mathcal{C}_{k+1}, w \mapsto w.\emptyset$, where the action . of $\widetilde{S}_{k+1}$ on $\mathcal{C}_{k+1}$ is defined by: 
    \begin{itemize}
        \item $s_i.\kappa$ is $\kappa$ with all $\kappa$-addable corners of $(k+1)$-residue $i$ added if there is at least one $\kappa$-addable corner of $(k+1)$-residue $i$,
        \item $s_i.\kappa$ is $\kappa$ with all $\kappa$-addable corners of $(k+1)$-residue $i$ removed if there is at least one $\kappa$-removable corner of $(k+1)$-residue $i$,  
        \item $s_i.\kappa$ is $\kappa$ if there are no $\kappa$-addable or $\kappa$-removable corner of $(k+1)$-residue $i$,
    \end{itemize}

    \item[3.] $\mathfrak{p}: \mathcal{C}_{k+1} \rightarrow \mathcal{P}_k, \kappa \mapsto \lambda$, where $\lambda_i$ is the number of cells in the $i$th row of $\kappa$ which have hook-length not exceed $k$.
    
\end{itemize}
We have $\mathfrak{p}^{-1}=\mathfrak{s}\circ \phi$. We define the {\em $k$-conjugate} of $\lambda \in \mathcal{P}_k$ by $\lambda^{(k)}= \mathfrak{p}( \mathfrak{p}^{-1}(\lambda)^t)$. 
\end{Defi}

\begin{example}\label{lambda_kappa_w}
\normalfont
Let $w$ be $s_{23043210}$ in Example \ref{wlambda}, and $\lambda, \kappa$ be partitions in Example \ref{lambda_kappa}. The diagram (\ref{bijection}) applying to $\lambda$ is $\label{bijection_example}
    \lambda \mapsto w \mapsto \kappa \mapsto \lambda$. Indeed, $\lambda \mapsto w$ and $w \mapsto \kappa$ are obtained by reading the green, and red tableaux below, respectively. We have $\lambda^{(k)}=(3,1,1,1,1,1)$.
    
    $$ \begin{array}{c|cccccccccc}
    &1&2&3&4&5&6&7&8&9\\
    \hline
     1&\color{green}{0}&\color{green}{1}&\color{green}{2}&\color{green}{3}&{4}&{0}&{1}&{2}&{3}\\
     2&\color{green}{4}&\color{green}{0}&{1}&{2}&{3}&{4}&{0}&{1}&{2}\\
     3&\color{green}{3}&{4}&{0}&{1}&{2}&{3}&{4}&{0}&{1}\\
     4&\color{green}{2}&{3}&{4}&{0}&{1}&{2}&{3}&{4}&{0}\\
     5&{1}&{2}&{3}&{4}&{0}&{1}&{2}&{3}&{4}\\
     6&{0}&{1}&{2}&{3}&{4}&{0}&{1}&{2}&{3}\\
     7&{4}&{0}&{1}&{2}&{3}&{4}&{0}&{1}&{2}\\
\end{array}
\quad
\quad
\begin{array}{c|cccccccccc}
    &1&2&3&4&5&6&7&8&9\\
    \hline
     1&\color{red}{0}&\color{red}{1}&\color{red}{2}&\color{red}{3}&\color{red}{4}&\color{red}{0}&{1}&{2}&{3}\\
     2&\color{red}{4}&\color{red}{0}&{1}&{2}&{3}&{4}&{0}&{1}&{2}\\
     3&\color{red}{3}&{4}&{0}&{1}&{2}&{3}&{4}&{0}&{1}\\
     4&\color{red}{2}&{3}&{4}&{0}&{1}&{2}&{3}&{4}&{0}\\
     5&{1}&{2}&{3}&{4}&{0}&{1}&{2}&{3}&{4}\\
     6&{0}&{1}&{2}&{3}&{4}&{0}&{1}&{2}&{3}\\
     7&{4}&{0}&{1}&{2}&{3}&{4}&{0}&{1}&{2}\\
\end{array}
$$

\end{example} 

\section{\texorpdfstring{$0$-}.Hecke algebras and nilCoxeter algebras}\label{A_k}
\begin{Defi}
\normalfont
Let $\mathcal{A}_k$ be the associative algebra over $\mathbb{Z}$ with generators $\{A_0,\dots,A_k\}$ satisfying relations (\ref{R2}), (\ref{R3}) (replace $s_i$ by $A_i$). It
is the group algebra of the affine braid group. The {\em $0$-Hecke algebra} $\mathcal{H}_k$ is the associative algebra over $\mathbb{Z}$ with generators $\{H_0,\dots,H_k\}$ satisfying relations $H_i^2 = -H_i$ for all $i$, and (\ref{R2}), (\ref{R3}) (replace $s_i$ by $H_i$) \cite{norton19790}. The {\em nilCoxeter algebra} $\mathcal{N}_k$ is the associative algebra over $\mathbb{Z}$ with generators $\{N_0,\dots,N_k\}$ satisfying relations $N_i^2 = 0$ for all $i$, and (\ref{R2}), (\ref{R3}) (replace $s_i$ by $N_i$) \cite{fomin1994schubert}. Then $\mathcal{H}_k$ and $\mathcal{N}_k$ are quotient algebras of $\mathcal{A}_k$. Namely, 
\begin{equation*}
      \mathcal{H}_k = \mathcal{A}_k/J_k \text{ and } \mathcal{N}_k = \mathcal{A}_k/I_k,
\end{equation*}
where $J_k$ (resp. $I_k$) is the two-sided ideal of $\mathcal{A}_k$ generated by all elements $A_i^2+A_i$ (resp. $A_i^2$).

An expression of $\alpha \in \mathcal{A}_k$ in generators is called a {\em word} of $\alpha$. The {\em weak length} $\widetilde{l}(\alpha)$ of $\alpha$ is defined to be the length of its word. Write $A_{i_1 \dots i_m}$ for $A_{i_1} \dots A_{i_m}$. If $u=A_{i_1 \dots i_m}$ is a word of $\alpha$, we call $\{i_1,\dots,i_m\}$ the {\em support} of $\alpha$, and denote $supp(\alpha)$. For a proper subset $S$ of $[0,k]$, we define the {\em canonical cyclic interval } 
\begin{equation*}\label{canonical_cyclic_interval}
    I_S: a+1 < \dots < k <  0 < \dots < a-1,
\end{equation*}where $a$ is the smallest integer in $[0,k]$ which is not in $S$. A word $u$ is called {\em $k$-connected} if $supp(u)$ is a proper subset of $[0,k]$, and also is an interval in $I_{supp(u)}$. We call $u=A_{i_1 \dots i_m}$ a {\em weak hook word} if with respect to the order $I_{supp(u)}$
\begin{itemize}
    \item[(a)]\label{a} $i_1 > \dots >i_j <i_{j+1} < \dots < i_m$, or \label{hook_a}
    \item[(b)]\label{b} $i_1 > \dots >i_j = i_{j+1} < \dots < i_m$, \label{hook_b}
\end{itemize}
and say {\em hook type} of $u$ is $V$ for (a), and $U$ for (b). The part with only inequalities $>$ (resp. $<$) is called {\em left side} (resp. {\em right side}) of hook form. Replace both $>,<$ by $<,>$ in (a), (b) give us {\em anti-weak hook words}. When $u$ is a weak hook word:
\begin{itemize}
    \item[(i)] The number of ascents $i_t<i_{t+1}$ of $u$ is denoted by $asc(u)$.
    \item[(ii)] With respect to order $I_{supp(u)}$, we consider all letters $a<c$ such that $a \ne c-1 \mod (k+1)$ and there is no letter $b$ in $u$ such that $a<b<c$. Let $u_{min}$ be the smallest possible value of $c$. Then by definition, the number of occurrences of $u_{min}$ in $u$ will be $0$, or $1$, or $2$. 
    \begin{enumerate}\label{ii}
        \item[(ii.1)] \label{ii1} If it is $0$, then $u$ is $k$-connected. 
        \item[(ii.2)] \label{ii2} If it is $1$, then $u$ is not $k$-connected and $u_{min}$ is on only one side of hook form.
        \item[(ii.3)] \label{ii3} If it is $2$, then $u$ is not $k$-connected and $u_{min}$ is on both sides of hook form.
    \end{enumerate}
    We say that $u$ is {\em $k$-weak connected} for the cases (ii.1), (ii.3).
\end{itemize}
Let $X_{i,con,side}^r$ be the set of all weak hook words $u$ in $\mathcal{A}_k$ such that hook type of $u$ is $X\in \{V,U\}$, $\widetilde{l}(u)=r$, $asc(u)=i$, connected type of $u$ is $con = \begin{cases}
c=\text{$k$-connected},\\ \overline{c}=\text{not $k$-connected}, \\
wc=\text{$k$-weak connected}, \\
\overline{wc}=\text{ not $k$-weak connected},
\end{cases}
$
and side of $u_{min}$ in hook form is $side \in \{\mathrm{left, right}\}$. Let $X^r_{i,con}$ be the set of all weak hook words $u$ of hook type $X$, $\widetilde{l}(u)=r$, $asc(u)=i$, connected type $con$. Let $X^r_{i}$ be the set of all weak hook words $u$ of hook type $X$, $\widetilde{l}(u)=r$, $asc(u)=i$.
\end{Defi}

\begin{example}
\normalfont
Let $k=4$. We have $supp(A_{0424}) = \{0,2,4\}$, $I_{\{0,2,4\}}=2<3<4<0$. So $A_{0424}$ is a weak hook word of type $V$, with only one ascent $2<4$, not $k$-connected, $(A_{0424})_{min} = 4$ is on both sides of hook form (the left side is $0>4>2$, the right side is $2<4$). Hence $A_{0424} \in V^4_{1,wc}$. Other examples are $A_{2240} \in U^4_{2,\overline{wc},\mathrm{right}}$, $A_{4224} \in U^4_{1,wc}$.
\end{example}

\begin{lemm}\label{left_right_cancel}
Let $X\in\{U,V\}$. We have
\begin{itemize}
    \item[1.] $X_{i}^r$ is the disjoint union of the sets $X_{i,wc}^r$, $X^r_{i,\overline{wc},\mathrm{left}}$, $X^r_{i,\overline{wc},\mathrm{right}}$,
    \item[2.] there is a bijection between $X_{i,\overline{wc},\mathrm{right}}$ and $X_{i-1,\overline{wc},\mathrm{left}}$ that preserves value in $\mathcal{A}_k$, 
    \item[3.] $X_{0,\overline{wc},\mathrm{right}}$, $V^r_{r-1,\overline{wc},\mathrm{left}}$, $U^r_{r-2,\overline{wc},\mathrm{left}}$, $X^r_i\, (i<0)$ are empty sets.
\end{itemize}
\end{lemm}
\begin{proof}\begin{itemize}
    \item[1.] Trivial by definition.
    \item[2.] For each $u \in X_{i,\overline{wc},\mathrm{right}}$, we see that $u_{min}$ commutes with the substring consisting of indices less than $u_{min}-1$ concerning the order $I_{supp(u)}$. Let $\tau(u)$ be the result after interchanging $u_{min}$ with this substring. Then $\tau(u) \in X_{i-1,\overline{wc},\mathrm{left}}$ and $\tau(u)=u$ in $\mathcal{A}_k$. The invert map of $\tau$ is still $\tau$, but from $X_{i-1,\overline{wc},\mathrm{left}}$ to $X_{i,\overline{wc},\mathrm{right}}$.
    \item[3.]Trivial by definition.
\end{itemize}
\end{proof}

\section{Anti-weak hook words to reduced hook words}\label{antihook}
\begin{Defi}
\normalfont
Let $\mathcal{M}$ be the monoid consisting of all words in the alphabet $\{a<a+1<\dots\}$. The {\em Coxeter-Knuth equivalence} on $\mathcal{M}$ is defined by relations
\begin{align}
    i(i+1)i &= (i+1)i(i+1) \text{ for all $i$,} \label{M1}\\
    kij &=kji \text{ for all $i<k\leq j, |i-j| \geq 2$,} \label{M2}\\
    ijk &=jik \text{ for all $i\leq k<j, |i-j| \geq 2$.} \label{M3}
\end{align}
Given a tableau $P$ with weakly increasing rows $P_1, \dots, P_l$ and $x_0 \in \mathcal{M}$, the {\em Edelman-Greene insertion} of $x_0$ into $P$ is defined as follows \cite[Definition 6.20]{edelman1987balanced}.
\begin{itemize}
    \item[1.] For each $i\geq 0$, if there is no entry in $P_{i+1}$ greater than $x_i$, then insert $x_i$ to the end of $P_{i+1}$ and stop. Otherwise, let $x_{i+1}$ be the first entry in $P_{i+1}$ greater than $x_i$.
    \item[2.] We replace $x_{i+1}$ in $P_{i+1}$ by $x_i$ if $x_i$ and $(x_i+1)$ do not appear together in $P_{i+1}$. 
    \item[3.] Repeat above steps with $x_{i+1}$ and $P_{i+2}$. The algorithm terminates when some $x_i$ is added to the end of a row. The result is a new tableau, which we denote $P \xleftarrow[]{EG} x_0$.
\end{itemize}
For a word $x=i_1 \dots i_m$ in $\mathcal{M}$, let $EG(x)$ be the tableau 
\begin{equation*}
    (\emptyset \xleftarrow[]{EG} i_1)\xleftarrow[]{EG} \dots \xleftarrow[]{EG} i_m.
\end{equation*}
For a given tableau $P$, let $\rho(P)$ be the word obtained by reading the rows of $P$ from bottom to top, and reading from left to right in each row.
\end{Defi}

\begin{example}\label{EGx}
\normalfont
Let $\mathcal{M}$ be the monoid with alphabet $\{1<2<3<4<5<6\}$. Let $x=24541$. The process to create $EG(x)$ is 
\begin{equation*}
    \emptyset \rightarrow 2 \rightarrow 24 \rightarrow 245 \rightarrow \begin{array}{ccc}
    2&4&5\\
    5&&
    \end{array} \rightarrow \begin{array}{ccc}
    1&4&5\\
    2&&\\
    5
    \end{array}
\end{equation*}
We have $\rho(EG(x)) = 52145$.
\end{example}
\begin{Defi}
\normalfont
Fix $k\in \mathbb{Z}_{\geq 0}$. Let $x$ be a word with $supp(x) \subsetneq \mathbb{Z}/(k+1)\mathbb{Z}$, and $\mathcal{M}_{x}$ be the monoid with alphabet $I_{supp(x)}$.
\end{Defi}
\begin{lemm}[Lemma 6.23, \cite{edelman1987balanced}]\label{EG_equiv}
Consider $x$ as a word in $\mathcal{M}_x$. If $s_x$ a reduced, then
\begin{itemize}
    \item[1.] $x$ and $\rho(EG(x))$ are Coxeter-Knuth equivalent, $s_x = s_{\rho(EG(x))}$.
    \item[2.] $EG(x)$ is a row and column strict tableau.
\end{itemize}
\end{lemm}

\begin{lemm}\label{xsAHN_reduced} Let $x$ be a word in $\mathcal{M}$.
\begin{itemize}
\item[1.]If one of $s_x, H_x, N_x$ is reduced, then the others are reduced.
\item[2.]If $s_x=s_y$ and they both are reduced, then $A_x=A_y$, and so $H_x=H_y, N_x=N_y$.
\end{itemize}
\end{lemm}
\begin{proof}
\begin{itemize}
    \item[1.] If one of $s_{x}, H_x, N_x$ is not reduced, then some $x'$ of form $\dots jj \dots$ obtained from $x$ by (\ref{M1}), (\ref{M2}), (\ref{M3}). It contradicts the fact that one of $s_x, H_x, N_x$ is reduced. 
    \item[2.] $s_x=s_y$ means $x,y$ are transformed to each other by (\ref{M1}), (\ref{M2}), (\ref{M3}), or (\ref{R1}). When they both are reduced, we do not apply (\ref{R1}). 
\end{itemize}
\end{proof}

\begin{lemm}\label{to_hook_word} 
Let $x$ be an anti-weak hook word in $\mathcal{M}_x$. There is a reduced hook word $H_{\overline{x}}$ (resp. $N_{\overline{x}}$) such that $H_x=H_{\overline{x}}$ (resp. $N_x = N_{\overline{x}}$), $|\text{left side of $x$}| = |\text{right side of $\overline{x}$}|$.
\end{lemm}
\begin{proof}
We construct $H_{\overline{x}}$ as follows.
\begin{itemize}
    \item[1.] For each $i$ appearing twice in $x$, we pick the last letter, interchange with the substring of letters greater than $i+1$ standing right before $i$, then replace all $ii$ by $i$. Let $\tilde{x}$ be the final result. Then $\tilde{x}$ is an anti-weak hook word in $\mathcal{M}_x$. 
    \item[2.] $H_{\tilde{x}}$ is a reduced word. Then by Lemma \ref{xsAHN_reduced}.1, $s_{\tilde{x}}$ is reduced, and then by Lemma \ref{EG_equiv}.1, $s_{\tilde{x}}=s_{\rho(EG(\tilde{x}))}$, and so by Lemma \ref{xsAHN_reduced}.2, $H_{\tilde{x}} = H_{\rho(EG(\tilde{x}))}$. Lemma \ref{EG_equiv}.2 says that $EG(\tilde{x})$ is a row and column strict tableau. Now, more than that, it has form 
    \begin{equation*}
        \begin{array}{ccccc}
             *&*&\dots&*&*\\
             \vdots&&&&\\
             *&&&&\\
        \end{array}
    \end{equation*}
    because $\tilde{x}$ is an anti-weak hook word. Then $\rho(EG(\tilde{x}))$ is a hook word in $\mathcal{M}_x$.
    \item[3.] Set $\overline{x}=\rho(EG(\tilde{x}))$, we can see 
    \begin{equation*}
        |\text{left side of $x$}| = |\text{left side of $\tilde{x}$}| = |\text{columns of $EG(\tilde{x})$}| = |\text{right side of $\overline{x}$}|.
    \end{equation*}
\end{itemize}
$H_{\overline{x}}$ is a word we are looking for. For $N_x$, we assume it is reduced and skip step 1.
\end{proof}
\begin{example}
\normalfont
Let $k=6$, and $x=2455421$. Then $\mathcal{M}_x$ is the monoid with alphabet $\{1<2<3<4<5<6\}$. We can see $x$ is an anti-weak hook word in $\mathcal{M}_x$. The words in the proof of Lemma \ref{to_hook_word} are $\tilde{x} = 24541$, $\overline{x}=52145$ by Example \ref{EGx}. The left side of $x$ is $2<4<5$ and the right side of $\overline{x}$ is $1<4<5$.
\end{example}

\section{Noncommutative symmetric functions}\label{noncomfunction}
\begin{Defi}
\normalfont
Fix $0\leq r \leq k$. Let $A$ be a proper subset of $r$ elements in $\mathbb{Z}/(k+1)\mathbb{Z}$. Set $d_A = A_{i_1  \dots i_r}$ and $i_A=A_{i_r \dots i_1}$, where $(i_1,\dots,i_r)$ is an arrangement of $A$ such that if $i,i+1 \in A$, then $i+1$ occurs before $i$. We call $d_A$ (resp. $i_A$) the {\em cyclically decreasing} (resp. {\em increasing}) {\em element} associated to $A$. The elements do not depend on the choice of $(i_1, \dots, i_r)$. 
\end{Defi}
\begin{example}\label{dAiA}
\normalfont
Let $k=4, A=\{0,2,4\}$. Then $d_A = A_{042} = A_{024}$, $i_A =A_{240} = A_{420}$. 
\end{example}
\begin{Defi}
\normalfont
Denote $\binom{\mathfrak{S}}{r}$ the set of all $r$-combinations of a set $\mathfrak{S}$. Following \cite{lam2006affine}, we define the {\em noncommutative homogenous symmetric functions} 
\begin{equation}\label{h_r}
h_r = \sum\limits_{A \in \binom{[0,k]}{r}} d_A,
\end{equation}
the {\em noncommutative elementary symmetric functions}
\begin{equation}\label{e_r}
e_r = \sum\limits_{A \in \binom{[0,k]}{r}} i_A,
\end{equation}
the {\em noncommutative hook Schur functions}
\begin{equation}\label{s_r_i}
s_{(r-i,1^i)} =  \sum\limits_{j=0}^{i} (-1)^j h_{r-i+j}e_{i-j},
\end{equation}
and the {\em noncommutative power sum symmetric functions}
\begin{equation}\label{p_r}
p_r = \sum\limits_{i=0}^{r-1} (-1)^i s_{(r-i,1^i)}.
\end{equation}
\end{Defi}
We need the following lemmas to get a cancellation-free expression of $p_r$. First, to simplify formulas, we write $\underline{S}$ for $\sum\limits_{u\in S} u$. Obviously, 
\begin{equation}\label{underline_sum}
 \underline{S\cup B} = \underline{S}+\underline{B} \text{ if } S\cap B =\emptyset.   
\end{equation}

\begin{lemm}\label{he_hook}
For $0 \leq i \leq r \leq k$, we have $h_{r-i}e_i = \underline{V_i^r} + \underline{V_{i-1}^r} + \underline{U_{i-1}^r}$.
\end{lemm}
\begin{proof}
By (\ref{h_r}), (\ref{e_r}), we have 
\begin{equation*}
    h_{r-i}e_i = \sum\limits_{A \in \binom{[0,k]}{r-i},\, B \in \binom{[0,k]}{i}} d_Ai_B. 
\end{equation*}
Suppose that $d_A = A_{i_1 \dots i_{r-i}}$, $i_B = A_{i_{r-i+1} \dots i_r}$. Rearranging $(i_1, \dots, i_{r-i})$ in decreasing order, and $(i_{r-i+1}, \dots, i_r)$ in increasing order with respect to $I_{A\cup B}$ do not change $d_A$, $i_B$. So we can assume that $i_1>\dots>i_{r-i}$ and $i_{r-i+1} < \dots <i_r$ with respect to $I_{A \cup B}$. Set $u=A_{i_1\dots i_r}$. Then $u \in V_i^r, V_{i-1}^r, U_{i-1}^r$ if $i_{r-i}$ is less, more, equal $i_{r-i+1}$, respectively.
\end{proof}

\begin{lemm}\label{s_hook}
For $0 \leq i \leq r \leq k$, we have
\begin{equation}\label{s_hook_reduced}
    s_{(r-i,1^i)} = \underline{V_i^r} + \sum\limits_{j=0}^{i-1} (-1)^j \underline{U^r_{i-j-1,wc}} + \underline{U^r_{i-1,\overline{wc},\mathrm{left}}}
\end{equation}
\end{lemm}
\begin{proof}
We have
\begin{align*}
    s_{(r-i,1^i)} &= \sum\limits_{j=0}^i (-1)^j h_{r-i+j}e_{i-j} &\text{ by  (\ref{s_r_i})}\\
    &= \sum\limits_{j=0}^i (-1)^j (\underline{V_{i-j}^r} + \underline{V_{i-j-1}^r} + \underline{U_{i-j-1}^r}) &\text{ by Lemma \ref{he_hook} }\\
    &= \underline{V_i^r} + \sum\limits_{j=0}^{i-1} (-1)^j \underline{U_{i-j-1}^r} &\\
    &= \underline{V_i^r} + \sum\limits_{j=0}^{i-1} (-1)^j (\underline{U_{i-j-1, wc}^r} + \underline{U_{i-j-1, \overline{wc}, \mathrm{left}}^r} + \underline{U_{i-j-1, \overline{wc}, \mathrm{right}}^r})&\text{ by (\ref{underline_sum}), Lemma \ref{left_right_cancel}.1}\\
    &= \underline{V_i^r} + \sum\limits_{j=0}^{i-1} (-1)^j \underline{U_{i-j-1, wc}^r} + \underline{U_{i-1, \overline{wc}, \mathrm{left}}^r} + (-1)^{i-1} \underline{U_{0, \overline{wc}, \mathrm{right}}^r} & \text{ by Lemma \ref{left_right_cancel}.2}\\
    &= \underline{V_i^r} + \sum\limits_{j=0}^{i-1} (-1)^j \underline{U_{i-j-1, wc}^r} + \underline{U_{i-1, \overline{wc}, \mathrm{left}}^r} & \text{ by Lemma \ref{left_right_cancel}.3}. 
\end{align*}
\end{proof}

\begin{lemm}\label{p_hook}
For $1\leq r \leq k$, we have  
\begin{equation}\label{p_hook_simple}
p_r = 
    \sum\limits_{i=0}^{r-1} (-1)^i \underline{V_{i,wc}^{r}} + \sum\limits_{i=1}^{r-1} (-1)^i(r-i) \underline{U_{i-1,wc}^{r}} + \sum\limits_{i=1}^{r-2} (-1)^i \underline{U_{i-1,\overline{wc},\mathrm{left}}^{r}}.
\end{equation}
\end{lemm}
\begin{proof}
We have 
\begin{align*}
    p_r &= \sum\limits_{i=0}^{r-1} (-1)^i s_{(r-i,1^i)} &\text{ by (\ref{p_r})}\\
    &= \sum\limits_{i=0}^{r-1} (-1)^i (\underline{V_i^r} + \sum\limits_{j=0}^{i-1}(-1)^j \underline{U_{i-j-1,wc}^r} + \underline{U_{i-1,\overline{wc},\mathrm{left}}^r}) &\text{ by (\ref{s_hook_reduced})} \\
    &= \sum\limits_{i=0}^{r-1}(-1)^i \underline{V_i^r} + \sum\limits_{i=0}^{r-1}\sum\limits_{j=0}^{i-1}(-1)^{i+j} \underline{U^r_{i-j-1,wc}} + \sum\limits_{i=1}^{r-1}(-1)^i \underline{U^r_{i-1,\overline{wc},\mathrm{left}}}.
\end{align*}
First
\begin{align*}
    \sum\limits_{i=0}^{r-1}(-1)^i \underline{V_i^r} &= \sum\limits_{i=0}^{r-1}(-1)^i(\underline{V^r_{i,wc}} + \underline{V^r_{i,\overline{wc},\mathrm{left}}} + \underline{V^r_{i,\overline{wc},\mathrm{right}}}) &\text{ by (\ref{underline_sum}), Lemma \ref{left_right_cancel}.1}\\
    &= \sum\limits_{i=0}^{r-1}(-1)^i \underline{V^r_{i,wc}} + \underline{V^r_{0,\overline{wc},\mathrm{right}}} + (-1)^{(r-1)}\underline{V^r_{r-1,\overline{wc},\mathrm{left}}} &\text{ by Lemma \ref{left_right_cancel}.2}\\
    &= \sum\limits_{i=0}^{r-1}(-1)^i \underline{V^r_{i,wc}} &\text{ by Lemma \ref{left_right_cancel}.3.}
\end{align*}
Second
\begin{align*}
    \sum\limits_{i=0}^{r-1}\sum\limits_{j=0}^{i-1}(-1)^{i+j} \underline{U^r_{i-j-1,wc}} &= \sum\limits_{t=1}^{r-1}\sum\limits_{j=0}^{r-t-1}(-1)^{t+2j} \underline{U^r_{t-1,wc}} &\text{ with $t=i-j$}\\
    &= \sum\limits_{t=1}^{r-1}(-1)^{t}(r-t) \underline{U^r_{t-1,wc}}.&
\end{align*}
Third
\begin{align*}
    \sum\limits_{i=1}^{r-1}(-1)^i \underline{U^r_{i-1,\overline{wc},\mathrm{left}}} &= \sum\limits_{i=1}^{r-2}(-1)^i \underline{U^r_{i-1,\overline{wc},\mathrm{left}}} &\text{ by Lemma \ref{left_right_cancel}.3.}
\end{align*}
So, the lemma is proved.
\end{proof}

\begin{Defi}
\normalfont
Let $\varphi$ be a representation of $\mathcal{A}_k$ on $\mathbb{C}[\widetilde{S}_{k+1}]$ and set $\alpha *_\varphi w = \varphi(\alpha)(w)$. A function $\psi: \mathcal{A}_{k} \times \widetilde{S}_{k+1} \rightarrow \mathbb{R}$ is said to be {\em $\varphi$-compatiable} if $\psi(\alpha\beta,w) = \psi(\alpha,\beta *_\varphi w) \psi(\beta, w)$. Fix a representation $\varphi$ and a $\varphi$-compatible function $\psi$, we define $\{\mathcal{F}_{w}^{(k)}\}_{w \in \widetilde{S}_{k+1}^0}$ to be a family of symmetric functions such that $\mathcal{F}_{id}^{(k)} = 1$ and 
\begin{align}
h_r . \mathcal{F}_w^{(k)} &= \sum\limits_{\substack{A \in \binom{[0,k]}{r} \\ d_A *_\varphi w \in \widetilde{S}_{k+1}^{0}}} \psi(d_A,w) \mathcal{F}_{d_A*_\varphi w}^{(k)}, \label{F_w^k_Pieri1}\\
e_r . \mathcal{F}_w^{(k)} &= \sum\limits_{\substack{B \in \binom{[0,k]}{r} \\ i_B *_\varphi w \in \widetilde{S}_{k+1}^{0}}} \psi(i_B,w) \mathcal{F}_{i_B*_\varphi w}^{(k)} \label{F_w^k_Pieri2},
\end{align}
for $w \in \widetilde{S}_{k+1}^0$, $0 \leq r \leq k$. Via (\ref{bijection}), we define $\mathcal{F}_\lambda^{(k)} = \mathcal{F}_{w_\lambda}^{(k)}$ for $\lambda \in \mathcal{P}_k$.
\end{Defi}

\section{A generalized Murnaghan-Nakayama rule}\label{genMN}
\begin{theo}\label{MN_Fwk} Let $\varphi$ be a representation of
$\mathcal{A}_k$ on $\mathbb{C}[\widetilde{S}_{k+1}]$ and let
$\{\mathcal{F}_w^{(k)}\}_{w \in \widetilde{S}^0_{k+1}}$ be the family
of symmetric functions whose Pieri rules are defined by
$\varphi$-compatible function $\psi$. Furthermore, suppose that $\psi(\alpha, w)$ depends only on $\widetilde{l}(\alpha)$, $\alpha *_\varphi w$, $w$, and we can write it as a function $\widetilde{\psi}$ on the three variables. Then for $1 \leq r \leq k$ and $w \in \widetilde{S}_{k+1}^{0}$, the product $p_r.\mathcal{F}_{w}^{(k)}$ is equal to
\begin{equation}\label{mn_Fwk}
     \sum\limits_{\substack{w' \in \widetilde{S}_{k+1}^{0}}} \widetilde{\psi}(r,w',w)\left( \sum\limits_{i=0}^{r-1} (-1)^i|V_{i,wc}^{r,w'}| + \sum\limits_{i=1}^{r-1} (-1)^i(r-i) |U_{i-1,wc}^{r,w'}| + \sum\limits_{i=1}^{r-2} (-1)^i |U_{i-1,\overline{wc},\mathrm{left}}^{r,w'}| \right)    \mathcal{F}_{w'}^{(k)},
\end{equation}
where $X_{i,con,side}^{r,w'}$ is the subset of all weak hook words $u$ in $X_{i,con,side}^r$ such that $u *_\varphi w = w'$.
\end{theo}
\begin{proof}
From (\ref{s_r_i}), (\ref{p_r}) we have
\begin{align}
    p_r &= \sum\limits_{i=0}^{r-1}\sum\limits_{j=0}^i (-1)^{i+j}h_{r-i+j}e_{i-j} \\
    &= \sum\limits_{i=0}^{r-1}\sum\limits_{j=0}^i\sum\limits_{\substack{A \in \binom{[0,k]}{r-i+j} \\ B \in \binom{[0,k]}{i-j}}}  (-1)^{i+j} d_Ai_B.\label{pr}
\end{align}
By (\ref{F_w^k_Pieri1}), (\ref{F_w^k_Pieri2}) we have
\begin{align}
    p_r.\mathcal{F}_w^{(k)} &= \sum\limits_{i=0}^{r-1}\sum\limits_{j=0}^i \sum\limits_{\substack{A \in \binom{[0,k]}{r-i+j} \\ d_A*_\varphi (i_B*_\varphi w) \in \widetilde{S}_{k+1}^{0}}} 
    \sum\limits_{\substack{B \in \binom{[0,k]}{i-j} \\  i_B*_\varphi w \in \widetilde{S}_{k+1}^{0}}} (-1)^{i+j}
    \psi(d_A,i_B*_\varphi w) \psi(i_B,w) \mathcal{F}_{d_A*_\varphi (i_B*_\varphi w)}^{(k)} \\
    &= \sum\limits_{i=0}^{r-1}\sum\limits_{j=0}^i \sum\limits_{\substack{
    A \in \binom{[0,k]}{r-i+j}\\
    B \in \binom{[0,k]}{i-j}\\
    (d_Ai_B)*_\varphi w \in \widetilde{S}^0_{k+1} 
    }} (-1)^{i+j}
    \psi(d_Ai_B,w) \mathcal{F}_{(d_Ai_B)*_\varphi w}^{(k)}\\
    &= \sum\limits_{i=0}^{r-1}\sum\limits_{j=0}^i \sum\limits_{\substack{
    A \in \binom{[0,k]}{r-i+j}\\
    B \in \binom{[0,k]}{i-j}\\
    (d_Ai_B)*_\varphi w \in \widetilde{S}^0_{k+1} 
    }} (-1)^{i+j}
    \widetilde{\psi}(r,(d_Ai_B)*_\varphi w, w) \mathcal{F}_{(d_Ai_B)*_\varphi w}^{(k)}.\label{prFw}
\end{align}
The identities (\ref{pr}), (\ref{prFw}) tell us that if $p_r = \sum\limits_u a_u u$ with $a_u \in \mathbb{C}$, $u$ has form $d_Ai_B$, then
\begin{equation}\label{prFw_reduced}
    p_r.\mathcal{F}_w^{(k)}=\sum\limits_{u*_\varphi w \in \widetilde{S}^0_{k+1}} \widetilde{\psi}(r,u*_\varphi w,w) a_u \mathcal{F}^{(k)}_{u *_\varphi w}.
\end{equation}
So the conclusion (\ref{mn_Fwk}) follows directly from (\ref{prFw_reduced}), Lemma \ref{p_hook}.
\end{proof}

\section{Murnaghan-Nakayama rule for \texorpdfstring{$K$-$k$-}.Schur functions}\label{MNKk}
\begin{Defi}
\normalfont
We consider {\em star representation} of $\mathcal{A}_k$ on $\mathbb{C}[\widetilde{S}_{k+1}]$ \cite{takigiku2019pieri, morse2012combinatorics} 
\begin{equation}\label{Demazure_rep}
    \eta(A_i)(w)=A_i * w =\begin{cases}
    s_iw &\text{ if }l(s_iw)>l(w),\\
    w &\text{ if }l(s_iw)<l(w).  
    \end{cases}
\end{equation}
Define $\psi: \mathcal{A}_k \times \widetilde{S}_{k+1} \rightarrow \mathbb{R}$ by $\psi(\alpha, w) = (-1)^{\widetilde{l}(\alpha)-l(\alpha * w)+l(w)}$. Then $\psi$ is $\eta$-compatible. The family $\{\mathcal{F}_w^{(k)}\}_{w \in \widetilde{S}^0_{k+1}}$ in this case coincides with the family of $K$-$k$-Schur functions $\{g_{w}^{(k)}\}_{w \in \widetilde{S}^0_{k+1}}$ \cite{takigiku2019pieri, morse2012combinatorics}. Since $\{\eta(A_i)\}_{i\in \mathbb{Z}/(k+1)\mathbb{Z}}$ satisfy relations $\eta(A_i)^2 = \eta(A_i)$, and (\ref{R2}), (\ref{R3}) (replace $s_i$ by $\eta(A_i)$), $\eta$ well defines a representation of $\mathcal{H}_k$ on $\mathbb{C}[\widetilde{S}_{k+1}]$ 
\begin{equation*}\label{Demazure_rep_induce}
    H_i = A_i J_k \mapsto -\eta(A_i).
\end{equation*}
\end{Defi}
Theorem \ref{MN_Fwk} taken from $\mathcal{H}_k$ gives us the Murnaghan-Nakayama rule for $K$-$k$-Schur functions.
\begin{coro}\label{MN_gwk}
For $1 \leq r \leq k$ and $w \in \widetilde{S}_{k+1}^{0}$, the product $p_r.g_{w}^{(k)}$ is equal to 
\begin{equation}\label{mn_gwk}
    \sum\limits_{\substack{w' \in \widetilde{S}_{k+1}^{0}}} (-1)^{r-l(w')+l(w)} \left( \sum\limits_{i=0}^{r-1} (-1)^i|V_{i,wc}^{r,w'}| + \sum\limits_{i=1}^{r-1} (-1)^i(r-i) |U_{i-1,wc}^{r,w'}| + \sum\limits_{i=1}^{r-2} (-1)^i |U_{i-1,\overline{wc},\mathrm{left}}^{r,w'}| \right) g_{w'}^{(k)}.
\end{equation}
\end{coro}

\begin{Defi}
\normalfont
Suppose that the diagram (\ref{bijection}) applying to $w \in \widetilde{S}^{0}_{k+1}$ is $
    w \mapsto \kappa \mapsto \lambda$. 
To translate Corollary \ref{MN_gwk} to $\{g_{\lambda}^{(k)}\}_{\lambda \in \mathcal{P}_k}$, we need a translation of (\ref{Demazure_rep}) to $\mathbb{C}[\mathcal{C}_{k+1}]$ 
\begin{equation}\label{Hikappa}
    A_i*\kappa = \begin{cases}
    \kappa \text{ with all $\kappa$-addable corners of $(k+1)$-residue $i$ added if they exist,}\\
    
    \kappa \text{ if there exists $\kappa$-removable corners of $(k+1)$-residue $i$,}\\
    
    0 \text{ otherwise,}
    \end{cases}
\end{equation}
and then to $\mathbb{C}[\mathcal{P}_k]$
\begin{equation}\label{Hilambda}
    A_i*\lambda = \begin{cases}
    \lambda + e_t \text{ where $t$ is the lowest row of $A_i.\kappa/\kappa$ if $A_i.\kappa \supsetneq \kappa$,} \\
    \lambda \text{ if } A_i*\kappa = \kappa,\\
    0 \text{ otherwise.}
    \end{cases}
\end{equation}
The actions seem not to be mentioned precisely in \cite{takigiku2019pieri, morse2012combinatorics}, but they are simple adaptations of \cite[Definition 18, Proposition 22]{lapointe2005tableaux} from $s_i$ to $A_i$, or of \cite[Section 4.1]{morse2012combinatorics} from $\mathfrak{s}_i$ to $A_i$. In (\ref{Hikappa}), the second line, we can view the image $\kappa$ as $\kappa$ with all $\kappa$-removable corners of residue $i$ absorbed. And in (\ref{Hilambda}), the second line, we can view the image $\lambda$ as $\lambda$ with the last cell in the row $t$ of $\lambda$ absorbed. 
\end{Defi}

\begin{rema}\label{Ob}
\normalfont 
We observe the process $u$ acting on $\kappa$, when $u=A_{i_1\dots i_r}$ is a weak hook word with $$i_1>\dots>i_j \leq i_{j+1} < \dots<i_r.$$ 
Suppose that $u*\kappa \ne 0$, we can see:
\begin{itemize}
    \item[Ob1.] By (\ref{Hikappa}), if $A_i *\kappa \ne 0$, there must be some $\kappa$-addable or $\kappa$-removable corners of residue $i$.
    \item[Ob2.] When we apply $A_iA_{i+1}$ on $\kappa$, if one cell of residue $i$ is added, then it must appear directly below one $\kappa$-addable or $\kappa$-removable corner of residue $i+1$. 
    \item[Ob3.] Set $V=A_{i_{j+1} \dots i_r}$, then by Ob2, $V*\kappa/\kappa$ must be a vertical strip. 
    \item[Ob4.] Set $\kappa^V = V*\kappa$. Similar to Ob2, when we apply $A_{j+1}A_j$ on $\kappa^V$, if one cell of residue $j+1$ is added, then it must appear directly to the right of one $\kappa^V$-addable or $\kappa^V$-removable corner of residue $j$. 
    \item[Ob5.] Set $H=A_{i_1 \dots i_j}$, then by Ob4, $H*\kappa^V/\kappa^V$ must be a horizontal strip.
    \item[Ob6.] By Ob3 and Ob5, $u*\kappa/\kappa$ must be a ribbon.
    \item[Ob7.] By (\ref{Hikappa}), $supp(u)=supp(u*\kappa/\kappa) \cup \{ \text{residues of some $\kappa$-removable corners}\}$.
\end{itemize}
\end{rema}

\begin{rema}\label{ob}
\normalfont 
Observe the process $u$ acting on $\lambda$, suppose that $u*\lambda \ne 0$, we can see:
\begin{itemize}
    \item[ob1.] Let $V=V^{(p)} \dots V^{(1)}$ be the factorization of $V$ into maximal segments of consecutive indices. Then for $1\leq p'\leq p$, the word $V^{(p')}$ adds a connected vertical strip, say $\mathcal{S}^{(p')}$, to $V^{(p'-1)}\dots V^{(1)}*\lambda$. The strips are disjoint from each other. The reason is the same as \cite[Corollary 4.3]{bandlow2011murnaghan}, with a remark that the last letter of $V^{(p')}$ may add or absorb one cell to $V^{(p'-1)}\dots V^{(1)}*\lambda$. Let $\mathcal{E}^{(p')}$ be the set of cells absorbed to $V^{(p'-1)}\dots V^{(1)}*\lambda$ (at most one cell). Then $|\mathcal{S}^{(p')}|+|\mathcal{E}^{(p')}|=\widetilde{l}(V^{(p')})$.
    \item[ob2.] By ob1, we have 
    \begin{align*}
        ht(V*\lambda/\lambda) &= \sum\limits_{p' \text{ s.t. }\mathcal{S}^{(p')} \ne \emptyset}|\mathcal{S}^{(p')}|-1,\\
        &=\sum\limits_{p' \text{ s.t. }\mathcal{S}^{(p')} \ne \emptyset} \widetilde{l}(V^{(p')}) - |\mathcal{E}^{(p')}|-1\\
        &=\sum\limits_{p' \text{ s.t. }\mathcal{S}^{(p')} \ne \emptyset} \widetilde{l}(V^{(p')}) - |\mathcal{E}^{(p')}|-1 + \sum\limits_{p' \text{ s.t. }\mathcal{S}^{(p')} = \emptyset} \widetilde{l}(V^{(p')}) - |\mathcal{E}^{(p')}|\\
        &= \widetilde{l}(V) - \sum\limits_{p' \text{ s.t. }\mathcal{S}^{(p')} \ne \emptyset } |\mathcal{E}^{(p')}|- \sum\limits_{p' \text{ s.t. }\mathcal{S}^{(p')} = \emptyset } |\mathcal{E}^{(p')}| - \sum\limits_{p' \text{ s.t. }\mathcal{S}^{(p')} \ne \emptyset} 1\\
        &=\widetilde{l}(V) - \sum\limits_{p' \text{ s.t. }\mathcal{S}^{(p')} \ne \emptyset } |\mathcal{E}^{(p')}| - \sum\limits_{p' \text{ s.t. }\mathcal{S}^{(p')} = \emptyset}1  -  \sum\limits_{p' \text{ s.t. }\mathcal{S}^{(p')} \ne \emptyset} 1 
        \intertext{because if $\mathcal{S}^{(p')} = \emptyset$, then $|\mathcal{E}^{(p')}|=1$,}
        &= \widetilde{l}(V) - p- \sum\limits_{p' \text{ s.t. }\mathcal{S}^{(p')} \ne \emptyset } |\mathcal{E}^{(p')}|.
    \end{align*} 
    So \begin{equation}\label{htV_lambda_lambda}
        ht(V*\lambda/\lambda) \leq \widetilde{l}(V) - p.
    \end{equation}
    The equality in (\ref{htV_lambda_lambda}) happens if and only if $\mathcal{E}^{(p')} = \emptyset$ for all $p'$ such that $\mathcal{S}^{(p')} \ne \emptyset$.
    \item[ob3.] When applying $H$ on $V*\lambda$, the cells absorbed or added to $\lambda$ form a horizontal strip by Ob5, (\ref{Hilambda}). So, a letter $A_i$ of $H$ increases $ht(V*\lambda/\lambda)$ if and only if $A_{i+1}$ is a first letter of some word $V^{(p')}$. Denote $i^{(p')}$ the first index of $V^{(p')}$, then
    \begin{align*}
        ht(\mu/\lambda) &= ht(V*\lambda/\lambda) + \sum\limits_{p' \text{ s.t }i^{(p')}-1 \text{ is an index of }H} 1.
    \end{align*}
So, we have
    \begin{equation}\label{ht_mu_lambda}
        ht(\mu/\lambda) \leq ht(V*\lambda/\lambda)+p\\
        \leq \widetilde{l}(V).
    \end{equation}
    Both inequalities in (\ref{ht_mu_lambda}) become equalities if and only if
    \begin{itemize}
        \item[1.] $i^{(p')}-1$ is an index of $H$ for all $p'$, and
        \item[2.] $\mathcal{E}^{(p')} = \emptyset$ for all $p'$ such that $\mathcal{S}^{(p')} \ne \emptyset$.
    \end{itemize}
    The first condition implies that the hook type of $u$ must be $V$. Hence, there are some particular cases we do not have equality in (\ref{ht_mu_lambda})
    \begin{itemize}
        \item[1.] the hook type of $u$ is $U$, or
        \item[2.] there is $p'$ such that $i^{(p')}-1$ is not an index of $H$, or
        \item[3.] there is $p'$ such that $\mathcal{S}^{(p')} \ne \emptyset$, $\mathcal{E}^{(p')} \ne \emptyset$.
    \end{itemize}
    In particular, we have equality in (\ref{ht_mu_lambda}) for a word satisfying conditions
    \begin{itemize}
        \item[1.] it has hook type $V$, is a $k$-connected word, and 
        \item[2.] $\mathcal{E}^{(p')} = \emptyset$ for all $p'$.
    \end{itemize}
    \item[ob4.] Similar to ob3, we want to think $\widetilde{l}(H)$ as height of a skew shape. Let $\overline{H},\overline{V}$ be $H,V$ after replacing index $i$ by $\overline{i}=k+1-i$, respectively. From a trivial fact that $A_i*\lambda = \mu$ if and only if $A_{\overline{i}}*\lambda^{(k)} = \mu^{(k)}$, we have 
    \begin{equation*}
        \overline{H}\,\overline{V}*\lambda^{(k)}=\mu^{(k)}.
    \end{equation*}
    We can see that $\overline{H}\,\overline{V}$ is an anti-weak hook word. Let $\overline{u}$ be the reduced weak hook word constructed from $\overline{H}\,\overline{V}$ by Lemma \ref{to_hook_word}, then $\overline{u}$ equals $\overline{H}\,\overline{V}$ in $\mathcal{H}_k$, and \begin{equation*}
        \widetilde{l}(H)=\widetilde{l}(\overline{H})=\widetilde{l}(V')+1, 
    \end{equation*}
    where $V'$ is the subword of $\overline{u}$ defined as in Ob3.
    By (\ref{ht_mu_lambda}), we have
    \begin{equation}\label{ht_mut_lambdat}
        ht(\mu^{(k)}/\lambda^{(k)}) \leq \widetilde{l}(V') = \widetilde{l}(H) -1.
    \end{equation}
    \item[ob5.] By (\ref{ht_mu_lambda}), (\ref{ht_mut_lambdat}), we have $ht(\mu/\lambda) + ht(\mu^{(k)}/\lambda^{(k)}) \leq r-1$. In particular, we have the equality for a word satisfying conditions
    \begin{itemize}
        \item[1.]it has hook type $V$, is a $k$-connected word, and
        \item[2.]all $\mathcal{E}^{(p')} = \emptyset$, for both parts $V$ of $u$ and $V'$ of $\overline{u}$. 
    \end{itemize}
\end{itemize}
\end{rema}

\begin{coro}\label{MN_gwk_partition}
For $1 \leq r \leq k$ and $\lambda \in \mathcal{P}_k$, the product $p_r.g_{\lambda}^{(k)}$ is equal to 
\begin{equation}\label{mn_gwk_partition}
    \sum\limits_{\mu} (-1)^{r-|\mu|+|\lambda|} \left( \sum\limits_{i=0}^{r-1} (-1)^i|V_{i,wc}^{r,\mu}| + \sum\limits_{i=1}^{r-1} (-1)^i(r-i) |U_{i-1,wc}^{r,\mu}| + \sum\limits_{i=1}^{r-2} (-1)^i |U_{i-1,\overline{wc},\mathrm{left}}^{r,\mu}| \right) g_{\mu}^{(k)},
\end{equation}
where the sum runs over $\mu \in \mathcal{P}_k$ such that 
\begin{itemize}
    \item[(K0)] \label{K0}$\lambda \subset \mu$, $\lambda^{(k)} \subset \mu^{(k)}$,
    \item[(K1)] \label{K1}$|\mu/\lambda| \leq r$,
    \item[(K2)] \label{K2}$\mathfrak{p}^{-1}(\mu)/\mathfrak{p}^{-1}(\lambda)$ is a ribbon,
    \item[(K3)] \label{K3}$supp(\mathfrak{p}^{-1}(\mu)/\mathfrak{p}^{-1}(\lambda))$ is $k$-connected, or $|supp(\mathfrak{p}^{-1}(\mu)/\mathfrak{p}^{-1}(\lambda))| <r$,  
    \item[(K4)] \label{K4}$ht(\mu/\lambda) + ht(\mu^{(k)}/\lambda^{(k)}) < r$.
\end{itemize}
\end{coro}
\begin{proof}
(K0), (K1) follow directly from (\ref{Hilambda}). (K2) is Ob6. (K3) follows directly from Ob7 and the fact that $\mathfrak{p}^{-1}(\mu)=u*\kappa$ where $\kappa = \mathfrak{p}^{-1}(\lambda)$, with $u$ taken from the sets in (\ref{mn_gwk_partition}). Indeed, by Ob7,
\begin{itemize}
    \item[1.] If $supp(\mathfrak{p}^{-1}(\mu)/\mathfrak{p}^{-1}(\lambda)) \subsetneq supp(u)$, then its cardinality is less than $r$. 
    \item[2.] If $supp(\mathfrak{p}^{-1}(\mu)/\mathfrak{p}^{-1}(\lambda)) = supp(u)$, by (\ref{mn_gwk_partition}), there are some possibilities:
    \begin{itemize}
        \item $u$ is $k$-connected.
        \item $u$ is not $k$-connected, $u_{min}$ is on both side of hook form. Then $|supp(u)|<r$. 
        \item $u$ is not $k$-connected, $u_{min}$ is on only one side of hook form. Then $u$ must have hook type $U$, and so $|supp(u)|<r$.
    \end{itemize}
    
\end{itemize}

(K4) is ob5.
\end{proof}

\begin{rema}\label{Ob'}
\normalfont
Suppose that we know $\pi=u*\kappa \ne 0$, we can see:
\begin{itemize}
    \item[Ob1'.] By Ob4, let $c$ be a cell in $\pi/\kappa$ so that the cell directly on its left-hand still belongs to $\pi/\kappa$ and $c$ has residue $j$, then $j$ must be an index of $H$. 
    \item[Ob2'.] Among $\kappa^V$-removable corners of residue $j$, if there exists a cell so that the cell directly below does not belong to $\pi/\kappa^V$, then $j$ can be an index of $H$. Contrary, if there is no such cell, $j$ can not be an index of $H$. Indeed, if $j$ is an index of $H$ and there is no such cell, then $j-1$ must be an index of $H$ and appear before $j$ in $H$. It contradicts the fact that indices of $H$ are decreasing.
    \item[Ob3'.] By Ob3, let $c'$ be a cell in $\pi/\kappa$ so that the cell directly below still belong to $\pi/\kappa$ and $c'$ has residue $i+1$, then $i+1$ must be an index of $V$.
    \item[Ob4'.] Among $\kappa$-removable corners of residue $i$, if there exists a cell so that the cell directly on its right-hand does not belong to $\kappa^V/\kappa$, then $i$ can be an index of $V$. Contrary, if there is no such cell, $i$ can not be an index of $V$. Indeed, if $i$ is an index of $V$ and there is no such cell, then $i+1$ must be an index of $V$ and appear before $i$ in $V$. It contradicts the fact that indices of $V$ are increasing.
\end{itemize}
In the picture below, let $\kappa$ be the white partition bounded by black and green bold lines. Let $\kappa^V$ be the partition bounded by green bold lines. Let $\pi$ be the biggest partition. The cells in Ob1', Ob2', Ob3', Ob4' are marked 1' (yellow), 2', 3' (green), 4', respectively.  
$$
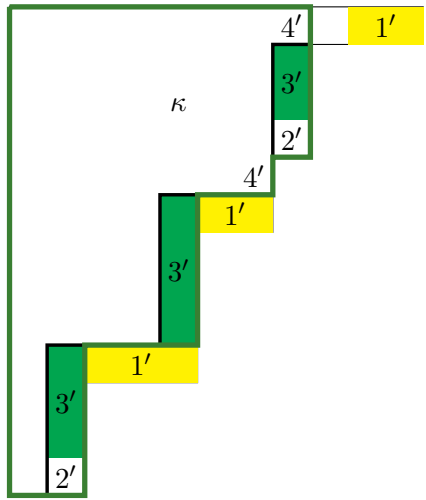
\hbox{
\begin{tikzpicture}
\draw (0,0)--++(11/2,0)--++(0,-1/2)--++(-3/2,0)--++(0,-3/2)--++(-1/2,0)--++(0,-2/2)--++(-2/2,0)--++(0,-4/2)--++(-3/2,0)--++(0,-3/2)--++(-2/2,0)--++(0,13/2);
\path[fill=Green] (7/2,-1/2)--++(1/2,0)--++(0,-2/2)--++(-1/2,0)--++(0,2/2);
\path[fill=Green] (4/2,-5/2)--++(1/2,0)--++(0,-4/2)--++(-1/2,0)--++(0,4/2);
\path[fill=Green] (1/2,-9/2)--++(1/2,0)--++(0,-3/2)--++(-1/2,0)--++(0,3/2);
\path[fill=yellow] (9/2,0)--++(2/2,0)--++(0,-1/2)--++(-2/2,0)--++(0,1/2);
\path[fill=yellow] (5/2,-5/2)--++(2/2,0)--++(0,-1/2)--++(-2/2,0)--++(0,1/2);
\path[fill=yellow] (2/2,-9/2)--++(3/2,0)--++(0,-1/2)--++(-3/2,0)--++(0,1/2);
\draw[line width=1.3pt](0,0)--++(8/2,0)--++(0,-1/2)--++(-1/2,0)--++(0,-4/2)--++(-3/2,0)--++(0,-4/2)--++(-3/2,0)--++(0,-4/2)--++(-1/2,0)--++(0,13/2);
\draw[line width=2pt, color = OliveGreen] (0,0)--++(8/2,0)--++(0,-4/2)--++(-1/2,0)--++(0,-1/2)--++(-2/2,0)--++(0,-4/2)--++(-3/2,0)--++(0,-4/2)--++(-2/2,0)--++(0,13/2);
\draw (0,0) node[below=1.3cm, right=2cm]{$\kappa$};
\draw (10/2,-1/4) node{$1'$};
\draw (3,-11/4) node{$1'$};
\draw (7/4,-19/4) node{$1'$};
\draw (15/4,-7/4) node{$2'$};
\draw (3/4,-25/4) node{$2'$};
\draw (15/4,-1) node{$3'$};
\draw (3/4,-21/4) node{$3'$};
\draw (9/4,-7/2) node{$3'$};
\draw (15/4,-1/4) node{$4'$};
\draw (13/4,-9/4) node{$4'$};
\end{tikzpicture}}
$$

\end{rema}
\begin{Defi}\label{Klambdamur}
\normalfont
For $\lambda,\mu$ satisfying the conditions in Corollary \ref{MN_gwk_partition}, we give a way to find the sets which contribute to (\ref{mn_gwk_partition}). It is based on Remark \ref{Ob'}. Namely, set $\kappa = \mathfrak{p}^{-1}(\lambda)$ and $\pi = \mathfrak{p}^{-1}(\mu)$, we construct a set $K_{\lambda,\mu}^r$ as follows.
\begin{itemize}
\item[Kk1.] Let $\mathsf{pre.must.down.cell}$ be the set of all cells $(i, j) \in \pi/\kappa$ such that $(i,j-1) \in \pi/\kappa$. The set of all $(k+1)$-residues of the cells in $\mathsf{pre.must.down.cell}$ is denoted by $\mathsf{must.down.residue}$. Let $\mathsf{must.down.cell}$ be the set of cells in $$ \pi/(d_{\mathsf{must.down.residue}})^{-1}*\pi.$$

\item[Kk2.] Let $\mathsf{pre.must.up.cell}$ be the set of all cells 
$(i, j) \in \pi/\kappa$ such that $(i+1,j) \in \pi/\kappa$. The set of all $(k+1)$-residues of the cells in $\mathsf{pre.must.up.cell}$ is denoted by $\mathsf{must.up.residue}$. Let $\mathsf{must.up.cell}$ be the set of cells in $$ i_{\mathsf{must.up.residue}}*\kappa/\kappa.$$

\item[Kk3.] Let $\mathsf{up.or.down.residue}$ be the set of all $(k+1)$-residues of the cells in $$(d_{\mathsf{must.down.residue}})^{-1}*\pi/i_{\mathsf{must.up.residue}}*\kappa.$$
Let $\mathsf{up.choice}$ be a subset of $\mathsf{up.or.down.residue}$. For each $\mathsf{up.choice}$, set \begin{equation*}
\mathcal{V} = \mathsf{must.up.residue} \cup \mathsf{up.choice}. 
\end{equation*}

\item[Kk3'.]For each $\mathcal{V}$ in step Kk3, let $\mathsf{absorb.cell.up}$ be the set of all $\kappa$-removable corners $(i,j)$ such that $(i, j+1) \not\in i_{\mathcal{V}}*\kappa/\kappa$. Set $\mathsf{absorb.residue.up}$ to be the set of all residues of the cells in $\mathsf{absorb.cell.up}$. For each subset $\mathsf{absorb.up.choice}$ of $\mathsf{absorb.residue.up}$ such that 
\begin{equation}\label{ine_less_r}
    |\mathsf{absorb.up.choice}| + |supp(\pi/\kappa)| \leq r,
\end{equation} set $V = i_{(\mathcal{V} \cup \mathsf{absorb.up.choice})}$.

\item[Kk4.] We have $V*\kappa = i_{\mathcal{V}}*\kappa$. It does not depend on the choice of $\mathsf{absorb.up.choice}$ in step Kk3'. 

\item[Kk5.] For each $V$ in step Kk3', let $\mathcal{H}$ be $supp(\pi/V*\kappa)$. It is exactly the union of $\mathsf{must.down.residue}$ with the complement of $\mathsf{up.choice}$ in $\mathsf{up.or.down.residue}$.

\item[Kk5'.] For each $V$ in step Kk3', let $\mathsf{absorb.cell.down}$ be the set of all $V*\kappa$-removable corners $(i,j)$ such that $(i+1,j) \not\in \pi/V*\kappa$. Set $\mathsf{absorb.residue.down}$ to be the set of all residues of the cells in $\mathsf{absorb.cell.down}$. For each subset $\mathsf{absorb.down.choice}$ of $\mathsf{absorb.residue.down}$ such that
\begin{equation}\label{equ_r}
    |\mathsf{absorb.down.choice}|+|\mathcal{H}|+|\mathcal{V}|+ |\mathsf{absorb.up.choice}| = r,
\end{equation} set $H = d_{(\mathcal{H} \cup \mathsf{absorb.down.choice})}$.

\item[Kk6.]
For each $H$ in step Kk5', set $u=HV$, rearrange indices of $H$ in decreasing order, of $V$ in increasing order with respect to $I_{supp(u)}$ and then add it to $K_{\lambda,\mu}^r$. This process does not change the value of $u$ in $\mathcal{A}_k$, but rewrites it as a weak hook word. Indeed,
\begin{itemize}
    \item $HV$ is not a weak hook word in general because $I_{supp(H)} \ne I_{supp(V)}$.
    \item The value of $V$ is not changed in $\mathcal{A}_k$. The reason is as follows. We have
    \begin{equation*}
        supp(V) \subset supp(u) \subset supp(I_{supp(u)}).
    \end{equation*}
    Let $a,b \in supp(V)$, then $a,b \in supp(I_{supp(u)})$. 
    We have
    \begin{itemize}
        \item[1.] if $b=a+1 \mod (k+1)$, then $a<b$ in $I_{supp(V)}$ implies $a<b$ in $I_{supp(u)}$,
        \item[2.] if $b \ne a+1 \mod (k+1)$, then $a<b$ in $I_{supp(V)}$ does not imply $a<b$ in $I_{supp(u)}$. However, we still have $A_{ab}=A_{ba}$ by (\ref{R3}).
    \end{itemize}
    In Kk3', we see that the indices of $V$ are increasing with respect to $I_{supp(V)}$. Let $V^{(p)} \dots V^{(1)}$ be the factorization of $V$ into maximal segments of consecutive indices in $I_{supp(V)}$. By points 1. and 2., the rearrangement on $V$ with respect to $I_{supp(u)}$ preserves segments $V^{(i)}\, (1\leq i \leq p)$, but permutes $(V^{(p)}, \dots, V^{(1)})$. If $a,b \in supp(V)$ are in two different segments, by definition of $V^{(i)}$, we know that $b\ne a \pm 1 \mod (k+1)$. By point 2., the value of $V=V^{(p)}\dots V^{(1)}$ is unchanged by any permutation on $(V^{(p)},\dots,V^{(1)})$. So it is preserved after the rearrangement of $V$ in increasing order with respect to $I_{supp(u)}$.  
    \item The value of $H$ is not changed in $\mathcal{A}_k$ by similar arguments above.
\end{itemize}
\end{itemize}
\end{Defi}

\begin{example}\label{K_lambda_mu_r_ex}
\normalfont
Let $k=4,r=4,\lambda=(4,2,1,1),\mu=(4,2,2,1,1,1)$. Then $\kappa = (6,2,1,1), \pi=(7,3,3,1,1,1)$. We can identify $\kappa, \pi/\kappa$ with the red, and blue tableaux below.
$$\begin{array}{c|cccccccccc}
    &1&2&3&4&5&6&7&8&9\\
    \hline
     1&\color{red}{0}&\color{red}{1}&\color{red}{2}&\color{red}{3}&\color{red}{4}&\color{red}{0}&\color{blue}{1}&{2}&{3}\\
     2&\color{red}{4}&\color{red}{0}&\color{blue}{1}&{2}&{3}&{4}&{0}&{1}&{2}\\
     3&\color{red}{3}&\color{blue}{4}&\color{blue}{0}&{1}&{2}&{3}&{4}&{0}&{1}\\
     4&\color{red}{2}&{3}&{4}&{0}&{1}&{2}&{3}&{4}&{0}\\
     5&\color{blue}{1}&{2}&{3}&{4}&{0}&{1}&{2}&{3}&{4}\\
     6&\color{blue}{0}&{1}&{2}&{3}&{4}&{0}&{1}&{2}&{3}\\
     7&{4}&{0}&{1}&{2}&{3}&{4}&{0}&{1}&{2}\\
\end{array}$$
\begin{itemize}
\item[Kk1.] $\mathsf{pre.must.down.cell} = \{(3,3)\}$,  $\mathsf{must.down.residue} = \{0\}$, $\mathsf{must.down.cell} = \{(6,1),(3,3)\}$. 

\item[Kk2.] $\mathsf{pre.must.up.cell} = \{(5,1),(2,3)\}$, $\mathsf{must.up.residue} = \{ 1 \}$, $\mathsf{must.up.cell} = \{(5,1),(2,3),(1,7)\}$. 

\item[Kk3.] $\mathsf{up.or.down.residue} = \{4\}$, $\mathsf{up.choice} \in \{\emptyset, \{4\}\}$. We have 
\begin{itemize}
    \item[1.] if $\mathsf{up.choice} = \emptyset$ then $\mathcal{V} = \{1\}$,
    \item[2.] if $\mathsf{up.choice} = \{4\}$ then $\mathcal{V} = \{1,4\}$.
\end{itemize}

\item[Kk3'.] Suppose that in step Kk3 we have $\mathsf{up.choice} = \{4\}$, $\mathcal{V} =\{1,4\}$, then $\mathsf{absorb.cell.up} = \{(4,1)\}$, $\mathsf{absorb.residue.up} = \{ 2\}$. So $\mathsf{absorb.up.choice} \in \{\emptyset, \{2\}\}$. Both of them satisfy the inequality (\ref{ine_less_r}). Now
\begin{itemize}
    \item[1.] if $\mathsf{absorb.up.choice} = \emptyset$ then $V = A_{14}$,
    \item[2.] if $\mathsf{absorb.up.choice} = \{2\}$ then $V = A_{124}$.
\end{itemize}

\item[Kk4.] For each $V$ in step Kk3', we have $V*\kappa = (7,3,2,1,1)$.

\item[Kk5.] For each $V$ in step Kk3', we have $\mathcal{H} = \{ 0 \}$.

\item[Kk5'.] For each $V$ in step Kk3', we have $\mathsf{absorb.cell.down} = \{(3,2), (1,7)\}$, $\mathsf{absorb.residue.down} = \{4,1\}$. So $\mathsf{absorb.down.choice} \in \{\emptyset, \{1\}, \{4\}, \{1,4\}\}$. To satisfy the equality (\ref{equ_r}), 
\begin{itemize}
    \item[1.] if $\mathsf{absorb.up.choice} = \emptyset$ then $\mathsf{absord.down.choice} = \{1\}$ or $\{4\}$, 
    \item[2.] if $\mathsf{absorb.up.choice} = \{2\}$ then $\mathsf{absord.down.choice} = \emptyset $.  
\end{itemize}
In the first case, $H=A_{10}$ or $A_{04}$, respectively. In the second case, $H=A_0$.

\item[Kk6.] The words $HV$ created in steps Kk3'--Kk5' are $A_{1014}, A_{0414}, A_{0124}$. Similarly, if in step Kk3, $\mathsf{up.choice} = \emptyset$, we create words $A_{1041}, A_{0412}$. Rewrite those words to weak hook form, we get $K_{\lambda,\mu}^r = \{A_{1041},A_{0441},A_{0412}\}$.
\end{itemize}
\end{example}

\begin{prop}\label{K_lambda_mu_r} $K_{\lambda,\mu}^r = \{u \in \mathcal{A}_k$ such that $u$ is a weak hook word, $\widetilde{l}(u)=r$, $u*\lambda = \mu \}$.
\end{prop}
\begin{proof}
Let $u=A_{i_1 \dots i_r}$ be a weak hook word in $\mathcal{A}_k$ of weak length $\widetilde{l}(u)=r$, with  $$i_1>\dots>i_j \leq i_{j+1} < \dots<i_r.$$ 
Let $H=A_{i_1 \dots i_j}$ and $V=A_{i_{j+1} \dots i_r}$. Set $\kappa = \mathfrak{p}^{-1}(\lambda)$, $\pi = \mathfrak{p}^{-1}(\mu)$. By Definitions \ref{def_bij}, $u*\lambda=\mu$ if and only if $u*\kappa =\pi$. After Remark \ref{Ob'}, the indices of $H,V$ will be classified into five cases below, and they are collected in the steps Kk1--Kk6 in the Definition \ref{Klambdamur} of $K_{\lambda,\mu}^r$: 
\begin{itemize}
    \item[(H1)] The residues mentioned in Ob1' must be indices of $H$. They must be added to $V*\kappa$.     
    \item[(H2)] The residues mentioned in Ob2' can be indices of $H$. If they are indices of $H$, then they must be absorbed to $V*\kappa$.
    \item[(V1)] The residues mentioned in Ob3' must be indices of $V$. They must be added to $\kappa$.  
    \item[(V2)] The residues mentioned in Ob4' can be indices of $V$. If they are indices of $V$, then they must be absorbed to $\kappa$.
    \item[(HV)] The residues that were not mentioned in Ob1'--Ob4' can be indices of $H$ or $V$. One element must be added to $\kappa$ as an index of $V$, or added to $V*\kappa$ as an index of $H$. In the picture of Remark \ref{Ob'}, the cells that were not mentioned in Ob1'--Ob4' are unmarked cells of $\pi/\kappa$.
\end{itemize}
Now we see (the numbering below corresponds to the numbering in Definition \ref{Klambdamur}):
\begin{itemize}
    \item[1.] The indices in case (H1) are collected in step Kk1 as $\mathsf{must.down.residue}$. The name of the set $\mathsf{must.down.residue}$ says the meaning ``must be an index of $H$".
    \item[2.] The indices in case (V1) are collected in step Kk2 as $\mathsf{must.up.residue}$. The name of the set $\mathsf{must.up.residue}$ says the meaning ``must be an index of $V$".
    \item[3.] The indices in case (HV) are collected in step Kk3. Indeed, the $(k+1)$-residues of cells in $$(d_{\mathsf{must.down.residue}})^{-1}*\pi/i_{\mathsf{must.up.residue}}*\kappa$$ can be indices of $H$ or $V$. We collect them in a set called $\mathsf{up.or.down.residue}$. Here, $\mathsf{down}$ means ``an index of $H$", $\mathsf{up}$ means ``an index of $V$". When we choose a subset $\mathsf{up.choice}$ of $\mathsf{up.or.down.residue}$ and combine it with the set $\mathsf{must.up.residue}$, we get a set $\mathcal{V}$ of all indices of $V$ that will be added to $\kappa$. 
    \item[3'.] So after step Kk3, the indices in case (V2) are collected in step Kk3'.  
    \item[4.] We have $i_{\mathcal{V}}*\kappa = V*\kappa$. 
    \item[5.] After step Kk3, the set of indices of $H$ added to $V*\kappa$ in cases (H1), (HV) is $\mathcal{H} = supp(\pi/V*\kappa)$. 
    \item[5'.] So after step Kk3', the indices in case (H2) are collected in step Kk5'. The condition $\widetilde{l}(u)=r$ is guaranteed by (\ref{ine_less_r}), (\ref{equ_r}).
    \item[6.] The condition $u$ is a weak hook word is handled in step Kk6 because the indices of $H$, $V$ are arranged in decreasing, increasing order with respect to $I_{supp(u)}$, respectively. 
\end{itemize}
\end{proof}

\begin{example}\label{coeff_K} 
\normalfont
In Example \ref{K_lambda_mu_r_ex}, we have $A_{1041} \in V_{1,wc}^{r,\mu}$, $A_{0441} \in U_{1,wc}^{r,\mu}$, $A_{0412} \in V_{2,wc}^{r,\mu}$. By Proposition \ref{K_lambda_mu_r}, (\ref{mn_gwk_partition}), the coefficient of $g^{(4)}_{(4,2,2,1,1,1)}$ in $p_4.g^{(4)}_{(4,2,1,1)}$ is $-2$.   
\end{example}

\section{Murnaghan-Nakayama rule for \texorpdfstring{$k$-}.Schur functions}\label{MNk}
\begin{Defi}
\normalfont
We consider {\em dot representation} of $\mathcal{A}_k$ on $\mathbb{C}[\widetilde{S}_{k+1}]$ \cite{lam2006affine} 
\begin{equation}\label{dot_rep}
    \zeta(A_i)(w)=A_i . w =\begin{cases}
    s_iw &\text{ if }l(s_iw)>l(w),\\
    0 &\text{ otherwise. }  
    \end{cases}
\end{equation}
Define $\psi: \mathcal{A}_k \times \widetilde{S}_{k+1} \rightarrow \mathbb{R}$ by $\psi(\alpha, w) = 1$. Then $\psi$ is $\zeta$-compatible. The family $\{\mathcal{F}_w^{(k)}\}_{w \in \widetilde{S}^0_{k+1}}$ in this case coincides with the family of $k$-Schur functions $\{s_{w}^{(k)}\}_{w \in \widetilde{S}^0_{k+1}}$ \cite{lapointe2007k,lam2006affine}. Since $\{\zeta(A_i)\}_{i\in \mathbb{Z}/(k+1)\mathbb{Z}}$ satisfy relations $\zeta(A_i)^2 = 0$, and (\ref{R2}), (\ref{R3}) (replace $s_i$ by $\zeta(A_i)$), $\zeta$ well defines a representation of $\mathcal{N}_k$ on $\mathbb{C}[\widetilde{S}_{k+1}]$ 
\begin{equation*}\label{dot_rep_induce}
    N_i = A_i I_k \mapsto \zeta(A_i).
\end{equation*}
\end{Defi}

Theorem \ref{MN_Fwk} taken from $\mathcal{N}_k$ gives us the Murnaghan-Nakayama rule for $k$-Schur functions. Furthermore, by identity $N_i^2=0$, the sets in (\ref{mn_Fwk}) will be simplified. Namely,
\begin{itemize}
    \item[1.] the sets of type $U$ become empty sets,
    \item[2.] the sets $V_{i,wc}^{r,w'}$ become $V_{i,c}^{r,w'}$. Indeed, let $\tau$ be the map in the proof of Lemma \ref{left_right_cancel}.2. For each $u$ satisfying condition (ii.3), we have $u=\tau(u)=0$ in $\mathcal{N}_k$.
\end{itemize}
We have proved the following result.
\begin{coro}[Theorem 3.1, \cite{bandlow2011murnaghan}]\label{MN_swk}
For $1 \leq r \leq k$ and $w \in \widetilde{S}_{k+1}^{0}$, we have 
\begin{equation}\label{mn_swk}
    p_r.s_{w}^{(k)} = \sum\limits_{\substack{w' \in \widetilde{S}_{k+1}^{0}}} \left( \sum\limits_{i=0}^{r-1} (-1)^i|V_{i,c}^{r,w'}| \right) s_{w'}^{(k)}.
\end{equation}
\end{coro}
\begin{Defi}
\normalfont
Suppose that the diagram (\ref{bijection}) applying to $w \in \widetilde{S}^{0}_{k+1}$ is $w \mapsto \kappa \mapsto \lambda$. To translate Corollary \ref{MN_swk} to $\{s_{\lambda}^{(k)}\}_{\lambda \in \mathcal{P}_k}$, we need a translation of (\ref{dot_rep}) to $\mathbb{C}[\mathcal{C}_{k+1}]$ 
\begin{equation}\label{Ai.kappa}
    A_i.\kappa = \begin{cases}
    \kappa \text{ with all $\kappa$-addable corners of $(k+1)$-residue $i$ added if they exist,}\\
    0 \text{ otherwise,}
    \end{cases}
\end{equation}
and then to $\mathbb{C}[\mathcal{P}_k]$
\begin{equation}\label{Ai.lambda}
    A_i.\lambda = \begin{cases}
    \lambda + e_t \text{ where $t$ is the lowest row of $N_i.\kappa/\kappa$ if $N_i.\kappa \supsetneq \kappa$,} \\
    0 \text{ otherwise.}
    \end{cases}
\end{equation}
The actions were mentioned in \cite{bandlow2011murnaghan,lam2006affine}, and they are simple adaptations of \cite[Definition 18, Proposition 22]{lapointe2005tableaux} from $s_i$ to $A_i$. We see that (\ref{Ai.kappa}), (\ref{Ai.lambda}) are (\ref{Hikappa}), (\ref{Hilambda}) without considering $\kappa$-removable corners, respectively. Hence, observations in Remarks \ref{Ob}, \ref{Ob'} for $u*\kappa$ are automatically transferred to $u.\kappa$ without considering $\kappa$-removable corners. And observations in Remark \ref{ob} for $u*\lambda$ are automatically transferred to $u.\lambda$ with all $\mathcal{E}^{(p')} =\emptyset$. Remember that the words $u$ we are considering are $k$-connected, and have hook type $V$. It immediately leads to the following results. The equivalence to \cite[Theorem 1.2]{bandlow2011murnaghan} will be explained in Remark \ref{sameTheo1.2}.
\end{Defi}
\begin{coro}[Theorem 1.2, \cite{bandlow2011murnaghan}]\label{MN_swk_partition}
For $1 \leq r \leq k$ and $\lambda \in \mathcal{P}_k$, we have  
\begin{equation}\label{mn_swk_partition}
    p_r.s_{\lambda}^{(k)} = \sum\limits_{\mu} \left( \sum\limits_{i=0}^{r-1} (-1)^i|V_{i,c}^{r,\mu}| \right) s_{\mu}^{(k)},
\end{equation}
where the sum runs over $\mu \in \mathcal{P}_k$ such that
\begin{itemize}
    \item[(k0)] \label{k0} $\lambda \subset \mu$, $\lambda^{(k)} \subset \mu^{(k)}$, 
    \item[(k1)] \label{k1} $|\mu/\lambda| = r$,
    \item[(k2)] \label{k2} $\mathfrak{p}^{-1}(\mu)/\mathfrak{p}^{-1}(\lambda)$ is a ribbon,
    \item[(k3)] \label{k3} $supp(\mathfrak{p}^{-1}(\mu)/\mathfrak{p}^{-1}(\lambda))$ is $k$-connected,   
    \item[(k4)] \label{k4} $ht(\mu/\lambda) + ht(\mu^{(k)}/\lambda^{(k)})=r-1$.
\end{itemize}
\end{coro}

\begin{Defi}
\normalfont
For $\lambda,\mu$ satisfying the conditions in Corollary \ref{MN_swk_partition}, we give a way to find the sets which contribute to (\ref{mn_swk_partition}). Namely, let $k_{\lambda,\mu}^r$ be a set constructed by the same steps in Definition \ref{Klambdamur} for $K_{\lambda,\mu}^r$ with $\mathsf{absorb.up.choice}$ in step Kk3' and $\mathsf{absorb.down.choice}$ in step Kk5' are empty sets. In particular, $k^r_{\lambda,\mu} \subset K^r_{\lambda,\mu}$. Because of (\ref{equ_r}), $k^r_{\lambda,\mu} \ne \emptyset$ if and only if $k^r_{\lambda,\mu} = K^r_{\lambda,\mu}$.  
\end{Defi}

\begin{prop}\label{k_lambda_mu_r} $k_{\lambda,\mu}^r = \{u \in \mathcal{A}_k$ such that $u$ is a weak hook word, $\widetilde{l}(u)=r$, $u.\lambda = \mu \}$.
\end{prop}

\begin{rema}\label{sameTheo1.2}
\normalfont
When $k_{\lambda,\mu}^{r}$ is nonempty, we have $|k_{\lambda,\mu}^{r}|=1$ (see \cite[Lemma 4.1(2)]{bandlow2011murnaghan}). Suppose that it contains $u$, then $ht(\mu/\lambda)=\widetilde{l}(V)=asc(u)$ by (\ref{ht_mu_lambda}). The coefficient of $s_{\mu}^{(k)}$ in (\ref{mn_swk_partition}) is
\begin{equation*}
    \sum\limits_{i=0}^{r-1} (-1)^i |V_{i,c}^{r,\mu}| = (-1)^{ht(\mu/\lambda)} + \sum\limits_{i \ne ht(\mu/\lambda)} (-1)^i |V_{i,c}^{r,\mu}| = (-1)^{ht(\mu/\lambda)}.
\end{equation*} This explains why Corollary \ref{MN_swk_partition} is equivalent to \cite[Theorem 1.2]{bandlow2011murnaghan}.
\end{rema}

\begin{example}\label{coeff_k}
\normalfont
With set up in Example \ref{K_lambda_mu_r_ex}, we have $k^r_{\lambda,\mu} =\emptyset$. This fits the fact that $\mu$ does not satisfy condition (k4). If we choose $\mu = (4,2,2,2,2)$, then $k^r_{\lambda,\mu} = \{A_{2134}\}$. Since $A_{2134} \in V^{r,\mu}_{2,c}$, then by (\ref{mn_swk_partition}), the coefficient of $s^{(k)}_{(4,2,2,2,2)}$ in $p_4.s^{(4)}_{(4,2,1,1)}$ is $1$.
\end{example}
\bibliography{references}{}
\bibliographystyle{alpha}

\noindent Institut für Algebra und Geometrie, Otto-von-Guericke-Universität Magdeburg, Germany.\\
E-mail: \href{khanh.mathematic@gmail.com}{khanh.mathematic@gmail.com} \\
\end{document}